\documentclass[11pt]{article}
\usepackage{amsmath,amssymb,amsthm}
\usepackage{float}
\usepackage{xcolor}
\usepackage{authblk}
\usepackage{autobreak}
\usepackage{enumerate}
\usepackage{hyperref}
\hypersetup{
	colorlinks,
	linkcolor={red!60!black},
	citecolor={green!60!black},
	urlcolor={blue!60!black}
}

\usepackage{tikz}
\usetikzlibrary{calc}

\definecolor{matchblue}{RGB}{40,60,180}
\definecolor{matchorange}{RGB}{230,130,0}
\definecolor{matchred}{RGB}{220,0,0}
\definecolor{matchgreen}{RGB}{0,150,70}
\definecolor{matchpink}{RGB}{235,170,180}

\newtheorem{theorem}{Theorem}[section]
\newtheorem{lemma}[theorem]{Lemma}
\newtheorem{corollary}[theorem]{Corollary}
\theoremstyle{definition}

\newtheorem{claim}[theorem]{Claim}
\newtheorem{conjecture}[theorem]{Conjecture}

\newtheorem*{repthm1}{Theorem~\ref{thm:main1}}
\newtheorem*{repthm2}{Theorem~\ref{thm:main2}}
\newtheorem*{repthm3}{Theorem~\ref{thm:main3}}
\newtheorem*{repthm4}{Theorem~\ref{thm:main4}}

\begin{document}
\title{Ramsey numbers and Gallai--Ramsey numbers of disjoint unions of cherries}
\author{Yanbo Zhang$^{1}$,\; Qian Chen$^{2}$,\; and Yaojun Chen$^{2}$\\
{\small $^1$School of Mathematical Sciences, Hebei Normal University, Shijiazhuang 050024, China}\\
{\small $^2$School of Mathematics, Nanjing University, Nanjing 210093, China}}
\date{}
\maketitle
\begin{quote}
{\bf Abstract:} For graphs $G_1,\ldots,G_k$, the Ramsey number $R(G_1,\ldots,G_k)$ is the smallest positive integer $N$ such that every $k$-edge-coloring of $K_N$ contains a monochromatic copy of $G_i$ in color $i$ for some $i\in[k]$. The Gallai--Ramsey number $GR(G_1,\ldots,G_k)$ is defined analogously, with the colorings restricted to Gallai colorings (i.e., edge-colorings with no rainbow triangle). 

A copy of $P_3$ is called a cherry. Let $n_iP_3$ denote the disjoint union of $n_i$ cherries. Wu, Magnant, Nowbandegani, and Xia (Discrete Appl. Math., 2019) proposed two conjectures:
\[
R(n_1P_3,\ldots,n_kP_3)=N\ \text{and}\ GR(n_1P_3,\ldots,n_kP_3)=N\,,
\]
where $N=2\max\{n_1,\ldots,n_k\}+\sum_{i=1}^kn_i-k+1$. We disprove the Ramsey conjecture and provide some sufficient conditions for determining the exact value of $R(n_1P_3,\ldots,n_kP_3)$. In contrast, we confirm the Gallai--Ramsey conjecture.


\end{quote}

\section{Introduction}

We denote by $[q,k]$ the integer interval $\{q,q+1,\ldots,k\}$, and write $[k]$ for $[1,k]$. 
Given a collection of graphs $G_1,\ldots,G_k$, the \emph{Ramsey number} $R(G_1,\ldots,G_k)$ is defined as the smallest positive integer $n$ such that every edge-coloring $\tau:E(K_n)\to[k]$ contains a monochromatic copy of $G_i$ in color $i$ for some $i\in[k]$. When $G_1=\cdots=G_k=G$, it is customary to abbreviate $R(G_1,\ldots,G_k)$ as $R_k(G)$.

A \emph{linear forest} is a forest whose connected components are nontrivial paths. If all its components have the same order, then the linear forest is called a \emph{path-matching}. Such a graph is usually denoted by $nP_m$, namely, the vertex-disjoint union of $n$ copies of the path $P_m$ on $m$ vertices.

In this paper, we are mainly concerned with the case $m=3$. A copy of $K_{1,2}\cong P_3$ is often called a \emph{cherry}; this terminology goes back at least to Erd\H{o}s and R\'enyi~\cite{ER63}, where it was used for two leaves adjacent to a common vertex. It is also used in the modern sense, for instance, by Sudakov and Volec~\cite{Sudakov2017}. Thus $nP_3$ denotes the vertex-disjoint union of $n$ cherries.

For two-color Ramsey numbers, Faudree and Schelp~\cite{Faudree1976} proved that
\[
R(n_1P_m,n_2P_m)=n_1m+n_2\left\lfloor \frac{m}{2}\right\rfloor-1
\]
for all $n_1\ge n_2$. In fact, they determined the exact two-color Ramsey numbers for all linear forests. For three colors, Scobee~\cite{Scobee1993} determined the Ramsey number
\[
R(m_1P_3,m_2P_3,m_3P_3)=3m_1+m_2+m_3-2
\]
for $m_1\ge m_2\ge m_3$ and $m_1\ge 2$, using a proof spanning about 20 pages.

For the Ramsey numbers of $nP_3$ in four or more colors, a complete determination currently appears to be very difficult. DeBiasio, Gy\'arf\'as, and S\'ark\"ozy~\cite{DeBiasio2021} turned to the more flexible problem of forcing a monochromatic linear forest of prescribed order, without specifying its exact structure. Given positive integers $k,n_1,\ldots,n_k$, let $R^{PM}(n_1,\ldots,n_k)$ denote the smallest integer $n$ such that every $k$-edge-coloring of $K_n$ contains a linear forest of color $i$ and order at least $n_i$ for some $i\in[k]$. They proved that for $k\ge 2$ and $n_1\ge \cdots \ge n_k\ge 2$, if $n_1\ge 2k-2$, then
\[
R^{PM}(n_1,\ldots,n_k)
= n_1+\sum_{i=2}^{k}\left\lceil \frac{n_i}{3}\right\rceil -k+1 .
\]

Cockayne and Lorimer~\cite{Cockayne1975} determined the multicolor Ramsey numbers for matchings of the form $nP_2$. More precisely, they proved that
\[
R(n_1P_2,\ldots,n_kP_2)=n+\sum_{i=1}^{k}n_i-k+1,
\]
where $n=\max\{n_1,\ldots,n_k\}$. Motivated by this lower-bound construction, Wu, Magnant, Nowbandegani, and Xia~\cite{Wu2019} replaced $P_2$ with $P_3$ and obtained the analogous lower bound
\[
R(n_1P_3,\ldots,n_kP_3)\ge 2n+\sum_{i=1}^{k}n_i-k+1,
\]
where $n=\max\{n_1,\ldots,n_k\}$. They further conjectured that this lower bound is tight.

\begin{conjecture}[Wu, Magnant, Nowbandegani, and Xia, 2019]\label{Wuconjecture2}
Let $k,n_1,\ldots,n_k$ be positive integers and let $n=\max\{n_1,\ldots,n_k\}$. Then
\[
R(n_1P_3,\ldots,n_kP_3)=2n+\sum_{i=1}^{k}n_i-k+1.
\]
\end{conjecture}

We disprove this conjecture by proving the following result.
\begin{theorem}\label{thm:main1}
  For positive integers $k, n_1, \ldots, n_k$, we have 
	\[R(n_1P_3, \ldots, n_kP_3) \ge 2\max\{\left\lceil k/2 \right\rceil, n_1, \ldots, n_k\} + \sum_{i=1}^{k} n_i-k+1.\]
\end{theorem}

Consequently, Conjecture~\ref{Wuconjecture2} fails whenever $k>2\max\{n_1,\ldots,n_k\}$.

By comparing $k$ with $2\max\{n_1,\ldots,n_k\}$, we see that Theorem~\ref{thm:main1} yields two different lower bounds. We next present two theorems showing that each of these two lower bounds is attained in certain ranges.

\begin{theorem}\label{thm:main2}
  For positive integers $k, n_1, \ldots, n_k$ with $3n_1\ge 4\sum_{i=2}^{k}n_i-2k+4$, we have 
  \[R(n_1P_3,\ldots, n_kP_3)=2n_1+\sum_{i=1}^{k} n_i-k+1.\]
\end{theorem}

\begin{theorem}\label{thm:main3}
\[
R_k(2P_3)=
\begin{cases}
2k+2, & \text{if}\ k\ge 9\ \text{and}\ k\ \text{is odd},\\[4pt]
2k+1, & \text{if}\ k\ge 14\ \text{and}\ k\ \text{is even}.
\end{cases}
\]
\end{theorem}

To prove Theorem~\ref{thm:main2}, we actually establish the following stronger statement. Here, $C_n$ denotes the cycle on $n$ vertices, and $K_{1,s}$ denotes the star with $s$ edges.

\begin{theorem}\label{thm:main4}
  For positive integers $k,s,n, n_2, \ldots, n_k$ with $s\ge 2$ and $n\ge 2\sum_{i=2}^{k}(n_i+1)s-6k+8$, we have 
  \[R(C_n, n_2K_{1,s}, \ldots, n_kK_{1,s})=n+\sum_{i=2}^{k} n_i-k+1.\]
\end{theorem}

Moreover, Irving~\cite{Irving1974} proved that $R_k(P_3)=k+1+(k \bmod 2)$. We also extend this result as follows.

\begin{theorem}\label{thm:main5}
\[
R_k(2P_3, P_3, \ldots, P_3)=
\begin{cases}
6, & \text{if}\ k=2,\\[2pt]
k+3, & \text{if}\ k\ \text{is odd},\\[2pt]
k+2, & \text{if}\ k\ \text{is even and}\ k\ge 4.
\end{cases}
\]
\[
R_k(3P_3, P_3, \ldots, P_3)=
\begin{cases}
9, & \text{if}\ k\in\{2,3,4\},\\[2pt]
k+4, & \text{if}\ k\ \text{is odd and}\ k\ge 5,\\[2pt]
11, & \text{if}\ k=6,\\[2pt]
k+3, & \text{if}\ k\ \text{is even and}\ k\ge 8.
\end{cases}
\]
\end{theorem}

It is worth noting that, when $k=6$, we have $R_6(3P_3,P_3,\ldots,P_3)=11$. This indicates that there are other possible lower-bound constructions, and hence the lower bound established in Theorem~\ref{thm:main1} is not tight in general and can still be improved.

Ramsey theory admits numerous variations, one of which restricts the allowable colorings and thereby often simplifies the multicolor setting. 
An edge-coloring $\tau:E(G)\to[k]$ is called a \emph{Gallai coloring} if it contains no rainbow triangle, that is, no triangle whose three edges receive pairwise distinct colors. 
This notion originates from work of Gallai~\cite{Gallai1967,Maffray2001} and is supported by a fundamental structural theorem.

\begin{theorem}[Gallai, 1967]\label{Gallai}
Let $(G,\tau)$ be a Gallai $k$-colored complete graph on at least two vertices. 
Then the vertex set of $G$ can be partitioned into nonempty parts $V_1,\ldots,V_p$ with $p\ge2$ such that at most two colors appear on the edges between different parts, and for each pair $(V_i,V_j)$, all edges between $V_i$ and $V_j$ receive the same color.
\end{theorem}

Motivated by this restriction, the \emph{Gallai--Ramsey number} $GR(G_1,\ldots,G_k)$ is defined as the smallest integer $N$ such that every Gallai coloring $\tau:E(K_N)\to[k]$ contains a monochromatic copy of $G_i$ in color $i$ for some $i\in[k]$. 
When $G_1=\cdots=G_k=G$, we simply write $GR_k(G)$.

We now consider Gallai--Ramsey numbers for disjoint unions of cherries. Owing to the structural restrictions imposed by Gallai colorings, these numbers admit a simple uniform formula, in contrast to the situation for general multicolor colorings.

Wu, Magnant, Nowbandegani, and Xia~\cite{Wu2019} studied the Gallai--Ramsey number of cherries $nP_3$. They provided upper and lower bounds for $GR(n_1P_3, \ldots, n_kP_3)$ and conjectured that the given lower bound is the exact value.

\begin{theorem}[Wu, Magnant, Nowbandegani, and Xia, 2019]\label{Wuconjecture1}
  Let $k, n_1, \ldots, n_k$ be positive integers and $n = \max\{n_1, \ldots, n_k\}$. We have 
	\[2n + \sum_{i=1}^{k} n_i - k + 1 \le GR(n_1P_3, \ldots, n_kP_3) \le (k+2)(n-1) + \frac{9n-3}{2}\ln\left(\frac{3n}{2}-1\right) + 1\,.\]
\end{theorem}

Although we have provided three different constructions yielding the lower bound for the Ramsey number $R(n_1P_3,\ldots,n_kP_3)$, only one of them is a Gallai coloring. Thus, the conjecture of Wu~\emph{et al.}\ is natural. Our last main result confirms their conjecture.

\begin{theorem}\label{Gallaitheorem}
  Let $k, n_1, \ldots, n_k$ be positive integers and $n = \max\{n_1, \ldots, n_k\}$. We have 
	\[GR(n_1P_3, \ldots, n_kP_3) = 2n + \sum_{i=1}^{k} n_i - k + 1\,.\]
\end{theorem}

It is worth noting that the proof of Theorem~\ref{Gallaitheorem} is inspired by the proofs of Zhang, Song, and Chen~\cite{Zhang2022,Zhang2023}. We believe that this approach will continue to be useful in the study of Gallai--Ramsey numbers.

The remainder of the paper is organized as follows. In Section~\ref{secforest0}, we prove our results on the Ramsey numbers of disjoint unions of cherries. In Section~\ref{secforest1}, we establish the corresponding Gallai--Ramsey numbers for disjoint unions of cherries.

\section{Ramsey numbers of disjoint unions of cherries}\label{secforest0}

We divide this section into five subsections, each of which restates one of the five theorems on Ramsey numbers and provides its proof.

\subsection{Proof of Theorem~\ref{thm:main1}}

\begin{repthm1}
  For positive integers $k, n_1, \ldots, n_k$, we have 
	\[R(n_1P_3, \ldots, n_kP_3) \ge 2\max\{\left\lceil k/2 \right\rceil, n_1, \ldots, n_k\} + \sum_{i=1}^{k} n_i-k+1.\]
\end{repthm1}

To prove Theorem~\ref{thm:main1}, we need Lemma~\ref{lem:Harary}. Moreover, Theorem~\ref{thm:main1} follows immediately from Lemmas~\ref{lem:main1.1} and~\ref{lem:main1.2}.

\begin{lemma}[Harary~\cite{Harary1969}]\label{lem:Harary}
  The complete graph $K_{2n}$ is $1$-factorable.
\end{lemma}

\begin{lemma}\label{lem:main1.1}
  For positive integers $k, n_1, \ldots, n_k$, we have 
	\[R(n_1P_3, \ldots, n_kP_3) \ge 2\left\lceil k/2 \right\rceil+ \sum_{i=1}^{k} n_i-k+1.\]
\end{lemma}

\begin{proof}
  Let $2\ell$ denote the smallest even number greater than or equal to $k$. In other words, $2\ell=k+(k \bmod 2)$. By Lemma~\ref{lem:Harary}, the complete graph $K_{2\ell}$ is $1$-factorable: it can be decomposed into $(2\ell-1)$ perfect matchings, each containing $\ell$ edges. We assign a unique color to each perfect matching, using at most $k$ colors, and ensuring that no monochromatic $P_3$ exists. Next, we incrementally add vertices to form a larger complete graph. First, we add $n_1-1$ vertices, denoted by $X$, and connect each of these vertices using only edges of color $1$. We then add $n_2-1$ vertices, connecting them with edges of color $2$. This process continues until we add $n_k-1$ vertices, each connected by edges of color $k$. At this point, the total number of vertices in the complete graph is \[2\ell+\sum_{i=1}^{k}n_i-k\,.\]

  Since every $P_3$ of color $1$ must include at least one vertex from $X$, constructing $n_1$ disjoint $P_3$'s of color $1$ requires $|X| \ge n_1$. However, $|X|=n_1-1$, leading to a contradiction. Thus, no $n_1P_3$ of color $1$ exists in this graph. Similarly, for each $i \in [k]$, no $n_iP_3$ of color $i$ exists. Therefore, we obtain
  \[R(n_1P_3, \ldots, n_kP_3) \ge 2\left\lceil k/2 \right\rceil+ \sum_{i=1}^{k} n_i-k+1.\qedhere\]
\end{proof}

\begin{lemma}\label{lem:main1.2}
  For positive integers $k, n_1, \ldots, n_k$, we have
  \[
    R(n_1P_3, \ldots, n_kP_3) \ge GR(n_1P_3, \ldots, n_kP_3) \ge 2\max\{n_1, \ldots, n_k\} + \sum_{i=1}^{k} n_i-k+1.
  \]
\end{lemma}

\begin{proof}
The first inequality is immediate from the definitions, since every Gallai coloring is an edge-coloring.

We now prove the second inequality. Without loss of generality, assume that $n_1=\max\{n_1,\ldots,n_k\}$. We construct a Gallai coloring of a complete graph on $N=2n_1+\sum_{i=1}^{k} n_i-k$ vertices that contains no monochromatic copy of $n_iP_3$ in color $i$ for any $i\in[k]$.

Start with a complete graph on $3n_1-1$ vertices whose edges are all colored with color $1$.  
For each $i=2,\ldots,k$, successively add $n_i-1$ new vertices, and color all edges incident with these new vertices by color $i$.  
No edges are recolored in the process.

By construction, this coloring contains no rainbow triangle, and hence is a Gallai coloring. Since only the initial $3n_1-1$ vertices are incident with edges of color $1$, the subgraph induced by color $1$ has fewer than $3n_1$ vertices, and therefore contains no monochromatic copy of $n_1P_3$.  
For each $i\in[2,k]$, deleting the $n_i-1$ vertices added at the $i$-th step eliminates all edges of color $i$.  
As the vertex cover number of $n_iP_3$ is $n_i$, it follows that the coloring contains no monochromatic copy of $n_iP_3$ in color $i$. This completes the proof.
\end{proof}

\subsection{Proof of Theorem~\ref{thm:main2}}
\begin{repthm2}
  For positive integers $k, n_1, \ldots, n_k$ with $3n_1\ge 4\sum_{i=2}^{k}n_i-2k+4$, we have 
  \[R(n_1P_3,\ldots, n_kP_3)=2n_1+\sum_{i=1}^{k} n_i-k+1.\]
\end{repthm2}

\begin{proof}
  The lower bound follows directly from Theorem~\ref{thm:main1}. 
  For the upper bound, we apply Theorem~\ref{thm:main4} by setting $n=3n_1$ and $s=2$.
  Note that the assumption $3n_1\ge 4\sum_{i=2}^{k}n_i-2k+4$ yields
  \[
    n=3n_1\ge 2\sum_{i=2}^{k}(n_i+1)s-6k+8,
  \]
  so the hypotheses of Theorem~\ref{thm:main4} are satisfied. Moreover,
  \[
    N:=2n_1+\sum_{i=1}^{k} n_i-k+1=n+\sum_{i=2}^{k} n_i-k+1.
  \]
  Consider an arbitrary $k$-edge-coloring of $K_N$. 
  If there is no monochromatic copy of $n_iP_3$ in color $i$ for some $i\in[2,k]$, then by Theorem~\ref{thm:main4} the graph contains a monochromatic cycle $C_n$ in color $1$. 
  Since $C_n$ contains $n_1P_3$ as a subgraph (recall that $n=3n_1$), we obtain a monochromatic copy of $n_1P_3$ in color $1$. 
  This completes the proof.
\end{proof}

\subsection{Proof of Theorem~\ref{thm:main3}}

We first establish four auxiliary lemmas, and then restate Theorem~\ref{thm:main3} and prove it.

\begin{lemma}\label{degree5}
Let $G$ be a graph that contains no subgraph isomorphic to $2P_3$. If there exists a vertex $v\in V(G)$ with $d(v)\ge 5$, then the edge set of $G-v$ is either empty or forms a matching.
\end{lemma}

\begin{proof}
We argue by contradiction. Suppose that the edge set of $G-v$ is neither empty nor a matching. Then $G-v$ contains a copy of $P_3$. Since $d(v)\ge 5$, the vertex $v$ has at least two neighbors that do not belong to this $P_3$. Consequently, $v$ together with two of its neighbors forms a $P_3$ that is vertex-disjoint from the previously identified $P_3$. Hence, $G$ contains $2P_3$ as a subgraph, contradicting the assumption. This completes the proof.
\end{proof}

\begin{lemma}\label{edgedegree}
Let $G$ be a connected graph with $e(G)\ge 11$ and $\Delta(G)\le 4$. Then $G$ contains $2P_3$ as a subgraph.
\end{lemma}

\begin{proof}
Let $P_s:=v_1v_2\cdots v_s$ be a longest path in $G$.

If $s=3$, then $G$ is either a triangle or a star. Since $\Delta(G)\le 4$, we have $e(G)\le 4$, contradicting the assumption that $e(G)\ge 11$.

If $s=4$, then neither $v_1$ nor $v_4$ has neighbors outside $P_4$, for otherwise the path could be extended. Moreover, $v_2$ and $v_3$ cannot both have neighbors outside $P_4$. Indeed, suppose that $v_2$ has a neighbor $u_2$ and $v_3$ has a neighbor $u_3$ outside $P_4$. If $u_2=u_3$, then $v_1v_2u_2v_3v_4$ is a longer path, a contradiction. If $u_2\neq u_3$, then $v_1v_2u_2$ and $u_3v_3v_4$ form two vertex-disjoint copies of $P_3$, yielding $2P_3$. If none of $\{v_1,v_2,v_3,v_4\}$ has neighbors outside $P_4$, then $e(G)\le 6$, again contradicting $e(G)\ge 11$. Hence, exactly one of $v_2$ and $v_3$ has neighbors outside $P_4$. Without loss of generality, assume that $v_2$ has such neighbors. Since $\Delta(G)\le 4$, $v_2$ has at most two neighbors outside $P_4$, say $u_2$ and $u_3$ (if they exist). As $P_4$ is a longest path, both $u_2$ and $u_3$ are adjacent only to $v_2$. It follows that $e(G)\le 5$, contradicting $e(G)\ge 11$.

If $s=5$, then neither $v_1$ nor $v_5$ has neighbors outside $P_5$, otherwise a longer path would exist. If $v_2$ has a neighbor $u_2$ outside $P_5$, then $v_1v_2u_2$ and $v_3v_4v_5$ form two vertex-disjoint copies of $P_3$. Hence, we may assume that $v_2$ has no neighbors outside $P_5$. By symmetry, we may also assume that $v_4$ has no neighbors outside $P_5$. If none of $\{v_1,v_2,v_3,v_4,v_5\}$ has neighbors outside $P_5$, then $e(G)\le 10$, contradicting $e(G)\ge 11$. Therefore, $v_3$ must have neighbors outside $P_5$. Since $\Delta(G)\le 4$, $v_3$ has at most two such neighbors, say $w_1$ and $w_2$ (if they exist). To avoid the existence of $2P_3$, each of $w_1$ and $w_2$ can have at most one additional neighbor, say $w_3$ and $w_4$, respectively (if they exist). If all of $w_1,w_2,w_3,w_4$ exist, then $w_3$ and $w_4$ must be distinct; otherwise, $w_1w_3w_2$ together with $v_1v_2v_3$ would form two vertex-disjoint copies of $P_3$. In this case, the graph $G-v_3$ contains exactly four edges, namely $v_1v_2$, $v_4v_5$, $w_1w_3$, and $w_2w_4$, which form a matching. Hence $e(G)\le 8$, again contradicting $e(G)\ge 11$. If some of the vertices $w_1,w_2,w_3,w_4$ do not exist, a straightforward verification shows that either $e(G)<11$, yielding a contradiction, or $G$ contains $2P_3$ as a subgraph.

If $s\ge 6$, then $v_1v_2v_3$ and $v_4v_5v_6$ form two vertex-disjoint copies of $P_3$.

In all cases, we conclude that $G$ contains $2P_3$ as a subgraph. This completes the proof.
\end{proof}

\begin{lemma}\label{eventheorem}
Let $k$ be an even integer with $k\ge 14$, and let $\ell$ be an integer with $1\le \ell \le k$.
Consider a $k$-edge-coloring of the complete graph $K_{k+\ell+1}$.
If there exist $k-\ell$ colors such that the edges of each of these colors form a matching,
then among the remaining $\ell$ colors there exists a monochromatic copy of $2P_3$.
\end{lemma}

\begin{proof}
For each $i\in[k]$, let $G_i$ denote the subgraph induced by the edges of color $i$.
Without loss of generality, assume that for each $i\in[\ell+1,k]$, the edge set of $G_i$ forms a matching.

We proceed by induction on $\ell$.
When $\ell=1$, we show that $G_1$ contains a copy of $2P_3$.
Suppose to the contrary that $G_1$ contains no $2P_3$.

If $\Delta(G_1)\ge 5$, let $v$ be a vertex of maximum degree in $G_1$.
By Lemma~\ref{degree5}, the edge set of $G_1-v$ is either empty or a matching.
Together with the assumption, this implies that in the $k$-edge-coloring of $K_{k+2}-v$ there is no monochromatic $P_3$.
Note that each vertex of $K_{k+2}-v$ is incident with exactly $k$ edges, and for each $i\in[2,k]$,
each vertex is incident with at most one edge of color $i$.
Hence every vertex is incident with exactly one edge of color $1$, that is, $G_1-v$ is a $1$-regular graph.
This is impossible since $G_1-v$ has an odd number of vertices.
Therefore, $\Delta(G_1)\le 4$.

Since each vertex of $K_{k+2}$ is incident with $k+1$ edges, and for each $i\in[2,k]$
each vertex is incident with at most one edge of color $i$, we have $\delta(G_1)\ge 2$.
If $G_1$ is disconnected, then from each connected component we can select a copy of $P_3$,
which yields a copy of $2P_3$ in $G_1$. Thus we may assume that $G_1$ is connected. Let $v$ be a vertex of maximum degree in $G_1$.

If $d_{G_1}(v)=2$, then $G_1$ is a connected $2$-regular graph, and hence a cycle of length at least $16$,
which clearly contains a copy of $2P_3$.
If $d_{G_1}(v)=4$, then $v$ has a non-neighbor $u$.
Let $u_1,u_2$ be two neighbors of $u$, and let $v_1,v_2$ be two neighbors of $v$ distinct from $u_1,u_2$.
Then $u_1uu_2$ and $v_1vv_2$ form two vertex-disjoint copies of $P_3$.
If $d_{G_1}(v)=3$, let $N_{G_1}(v)=\{v_1,v_2,v_3\}$.
If some $v_i$ has two neighbors outside $\{v,v_1,v_2,v_3\}$, then that vertex together with its two neighbors
forms a $P_3$, while $v_jvv_k$ (where $\{i,j,k\}=\{1,2,3\}$) forms another vertex-disjoint $P_3$.
Hence we may assume that each of $v_1,v_2,v_3$ has at most one neighbor outside $\{v,v_1,v_2,v_3\}$.
Consequently, there exists a vertex $w$ at distance at least $3$ from $v$.
Then $w$ together with any two of its neighbors forms a $P_3$, and $v_1vv_2$ forms another vertex-disjoint $P_3$.
In all cases, $G_1$ contains a copy of $2P_3$, completing the proof for the base case $\ell=1$.

Now assume that the statement holds for all smaller values of $\ell$, and consider $\ell\ge 2$.
Without loss of generality, suppose that $e(G_1)\ge e(G_i)$ for all $i\in[2,\ell]$.
Then
\[
e(G_1)\ge
\left\lceil
\frac{\binom{k+\ell+1}{2}-(k-\ell)\left\lfloor \frac{k+\ell+1}{2}\right\rfloor}{\ell}
\right\rceil
\ge k+\ell+1.
\]
We claim that $G_1$ contains a copy of $2P_3$.

If $\Delta(G_1)\ge 5$, let $v$ be a vertex of maximum degree in $G_1$.
By Lemma~\ref{degree5}, the edge set of $G_1-v$ is either empty or a matching.
Together with the assumption, this implies that in the $k$-edge-coloring of $K_{k+\ell+1}-v$
there are $k-\ell+1$ colors whose edges form matchings.
Note that $K_{k+\ell+1}-v=K_{k+(\ell-1)+1}$, and $k-\ell+1=k-(\ell-1)$.
By the induction hypothesis, among the remaining $\ell-1$ colors there exists a monochromatic copy of $2P_3$.
Hence we may assume that $\Delta(G_1)\le 4$.

Let $X$ be a largest connected component of $G_1$.
Since $e(G_1)\ge k+\ell+1$, we have $|X|\ge 3$.
We first claim that $e(G_1[X])\ge 11$.
Suppose, for a contradiction, that $e(G_1[X])\le 10$.
Then the total number of edges in the remaining components of $G_1$ is 
$e(G_1)-e(G_1[X])\ge (k+\ell+1)-e(G_1[X])$.
The number of vertices outside $X$ is $(k+\ell+1)-|X|$.
If $G_1$ contains a copy of $P_3$ outside $X$, then $G_1$ contains $2P_3$ and we are done.
Thus we may assume that $G_1$ contains no $P_3$ outside $X$, and hence the number of edges outside $X$
is at most $\left\lfloor (k+\ell+1-|X|)/2 \right\rfloor$.
Therefore,
\[
\left\lfloor \frac{k+\ell+1-|X|}{2} \right\rfloor
\ge (k+\ell+1)-e(G_1[X]).
\]
If $|X|\ge 5$, then together with $e(G_1[X])\le 10$ this yields
\[
\left\lfloor \frac{k+\ell+1-5}{2}\right\rfloor \ge k+\ell-9,
\]
which implies $k+\ell\le 14$, contradicting $k\ge 14$.
If $|X|=4$, then $e(G_1[X])\le 6$; if $|X|=3$, then $e(G_1[X])\le 3$.
Both cases similarly lead to contradictions.
Hence $e(G_1[X])\ge 11$.

Applying Lemma~\ref{edgedegree} to $G_1[X]$, we conclude that $G_1[X]$ contains a copy of $2P_3$.
This completes the proof of the theorem.
\end{proof}

The proof of the following lemma closely parallels that of Lemma~\ref{eventheorem}.
For the sake of completeness, we provide the full details.

\begin{lemma}\label{oddtheorem}
Let $k$ be an odd integer with $k\ge 9$, and let $\ell$ be an integer with $1\le \ell \le k$.
Consider a $k$-edge-coloring of the complete graph $K_{k+\ell+2}$.
If there exist $k-\ell$ colors such that the edges of each of these colors form a matching,
then among the remaining $\ell$ colors there exists a monochromatic copy of $2P_3$.
\end{lemma}

\begin{proof}
For each $i\in[k]$, let $G_i$ denote the subgraph induced by the edges of color $i$.
Without loss of generality, assume that for each $i\in[\ell+1,k]$, the edge set of $G_i$ forms a matching.

We proceed by induction on $\ell$.
When $\ell=1$, we prove that $G_1$ contains a copy of $2P_3$.
Suppose, to the contrary, that $G_1$ contains no $2P_3$.

If $\Delta(G_1)\ge 5$, let $v$ be a vertex of maximum degree in $G_1$.
Consider the complete graph $K_{k+3}-v$.
By the assumption, for each $i\in[2,k]$, every vertex in this graph is incident with at most one edge of color $i$.
Since each vertex of $K_{k+3}-v$ is incident with $k+1$ edges in total,
it follows that every vertex is incident with at least two edges of color $1$.
Hence $G_1-v$ contains a copy of $P_3$.
Choose two neighbors of $v$ that do not lie on this $P_3$; together with $v$ they form another copy of $P_3$.
This yields a copy of $2P_3$ in $G_1$, contradicting our assumption.
Therefore, $\Delta(G_1)\le 4$.

Since each vertex of $K_{k+3}$ is incident with $k+2$ edges, and for each $i\in[2,k]$
each vertex is incident with at most one edge of color $i$, we have $\delta(G_1)\ge 3$.
Fix a copy of $P_3$, say $v_1v_2v_3$, in $G_1$.
For any vertex $u\notin\{v_1,v_2,v_3\}$, if $u$ has two neighbors outside $\{v_1,v_2,v_3\}$,
then we obtain a copy of $P_3$ that is vertex-disjoint from $v_1v_2v_3$, a contradiction.
Hence every such vertex $u$ has at least two neighbors in $\{v_1,v_2,v_3\}$.
Consequently, there are at least $14$ edges between $\{v_1,v_2,v_3\}$ and $K_{k+3}\setminus\{v_1,v_2,v_3\}$,
which contradicts the assumption that $\Delta(G_1)\le 4$.
This completes the proof of the base case $\ell=1$.

Now assume that the statement holds for all smaller values of $\ell$, and consider $\ell\ge 2$.
Without loss of generality, suppose that $e(G_1)\ge e(G_i)$ for all $i\in[2,\ell]$.
Then
\[
e(G_1)\ge
\left\lceil
\frac{\binom{k+\ell+2}{2}-(k-\ell)\left\lfloor \frac{k+\ell+2}{2}\right\rfloor}{\ell}
\right\rceil
\ge k+\ell+4.
\]
We claim that $G_1$ contains a copy of $2P_3$.

If $\Delta(G_1)\ge 5$, let $v$ be a vertex of maximum degree in $G_1$.
By Lemma~\ref{degree5}, the edge set of $G_1-v$ is either empty or a matching.
Together with the assumption, this implies that in the $k$-edge-coloring of $K_{k+\ell+2}-v$
there are $k-\ell+1$ colors whose edges form matchings.
Note that $K_{k+\ell+2}-v$ is isomorphic to $K_{k+(\ell-1)+2}$ and that $k-\ell+1=k-(\ell-1)$.
By the induction hypothesis, among the remaining $\ell-1$ colors there exists a monochromatic copy of $2P_3$.
Hence we may assume that $\Delta(G_1)\le 4$.
Moreover, since $e(G_1)\ge k+\ell+4$, it follows that $\Delta(G_1)\ge 3$.

Let $X$ be a largest connected component of $G_1$.
Clearly, $|X|\ge 4$.
We claim that $e(G_1[X])\ge 11$.
Suppose, for a contradiction, that $e(G_1[X])\le 10$.
Then the total number of edges in the remaining components of $G_1$ is at least
$e(G_1)-e(G_1[X])\ge (k+\ell+4)-e(G_1[X])$.
The number of vertices outside $X$ is at most $(k+\ell+2)-|X|$.
If $G_1$ contains a copy of $P_3$ outside $X$, then $G_1$ contains $2P_3$ and we are done.
Thus we may assume that $G_1$ contains no $P_3$ outside $X$,
and hence the number of edges outside $X$ is at most
$\left\lfloor (k+\ell+2-|X|)/2 \right\rfloor$.
Therefore,
\[
\left\lfloor \frac{k+\ell+2-|X|}{2} \right\rfloor
\ge (k+\ell+4)-e(G_1[X]).
\]
If $|X|\ge 5$, then together with $e(G_1[X])\le 10$ this yields
\[
\left\lfloor \frac{k+\ell+2-5}{2}\right\rfloor \ge k+\ell-6,
\]
which implies $k+\ell\le 9$, contradicting the assumption that $k\ge 9$.
If $|X|=4$, then $e(G_1[X])\le 6$, which similarly leads to a contradiction.
Hence $e(G_1[X])\ge 11$.

Applying Lemma~\ref{edgedegree} to $G_1[X]$, we conclude that $G_1[X]$ contains a copy of $2P_3$.
This completes the proof of the theorem.
\end{proof}

\begin{repthm3}
\[
R_k(2P_3)=
\begin{cases}
2k+2, & \text{if}\ k\ge 9\ \text{and}\ k\ \text{is odd},\\[4pt]
2k+1, & \text{if}\ k\ge 14\ \text{and}\ k\ \text{is even}.
\end{cases}
\]
\end{repthm3}

\begin{proof}
The lower bound follows directly from Theorem~\ref{thm:main1}. If $k$ is even, the upper bound follows immediately from Lemma~\ref{eventheorem} by taking $\ell=k$. If $k$ is odd, it follows analogously from Lemma~\ref{oddtheorem} by taking $\ell=k$.
\end{proof}

\subsection{Proof of Theorem~\ref{thm:main4}}
\begin{lemma}[Bondy~\cite{Bondy1971}]\label{lem:Bondy}
Let $G$ be a graph on $n$ vertices. If $\delta(G)\ge n/2$, then $G$ contains a cycle of every length $\ell$ with $3\le \ell\le n$, unless $G$ is the complete bipartite graph $K_{n/2,n/2}$.
\end{lemma}

\begin{repthm4}
  For positive integers $k,s,n, n_2, \ldots, n_k$ with $s\ge 2$ and $n\ge 2\sum_{i=2}^{k}(n_i+1)s-6k+8$, we have 
  \[R(C_n, n_2K_{1,s}, \ldots, n_kK_{1,s})=n+\sum_{i=2}^{k} n_i-k+1.\]
\end{repthm4}
\begin{proof}
We first prove the lower bound. Consider a complete graph on $n-1$ vertices and color every edge among them with color~$1$. For each $i\in\{2,\ldots,k\}$, add $n_i-1$ new vertices one step at a time, and color every edge incident with these newly added vertices with color~$i$.

In this coloring, only the original $n-1$ vertices are incident with color~$1$ edges. Hence the color-$1$ subgraph has fewer than $n$ vertices and thus contains no monochromatic copy of $C_n$. Fix $i\in[2,k]$. Removing the $n_i-1$ vertices introduced at the $i$-th step deletes all edges of color~$i$. As the vertex cover number of $n_iK_{1,s}$ is $n_i$, it follows that the coloring contains no monochromatic copy of $n_iK_{1,s}$ in color $i$.

We now turn to the proof of the upper bound. Let $G$ be a $k$-edge-colored complete graph on $n+\sum_{i=2}^{k} n_i-k+1$ vertices.

For each $i\in[2,k]$, we find as many vertex-disjoint $K_{1,s}$ of color $i$ in $G$ as possible. Suppose that $n_i'$ such $K_{1,s}$ are obtained. Then $n_i'<n_i$, since otherwise there would exist $n_i$ vertex-disjoint copies of $K_{1,s}$ in color $i$, completing the proof. Let $X_i$ denote the vertex set of these $n_i'K_{1,s}$ of color $i$. For each $K_{1,s}$ of color $i$, we claim the following.
  
  \begin{claim}
  For every $K_{1,s}$ of color $i$, all but at most one of its vertices have at most $2s-2$ edges of color $i$ connecting to $V(G)\setminus X_i$.
  \end{claim}
  
  \begin{proof}
  We proceed by contradiction. Suppose there exists a $K_{1,s}$ of color $i$ with at least two vertices, say $u_1$ and $u_2$, each having at least $2s-1$ edges of color $i$ connecting to $V(G)\setminus X_i$.

  If $u_1$ and $u_2$ are both leaves of $K_{1,s}$, let $u_0$ denote the center of $K_{1,s}$. Then, we can select $s$ neighbors of $u_1$ from $V(G)\setminus X_i$, denoted by the set $U_1$, such that all edges between $u_1$ and $U_1$ are of color $i$. Similarly, we can select $s$ neighbors of $u_2$ from $(V(G)\setminus (X_i\cup U_1))\cup\{u_0\}$, denoted by the set $U_2$, such that all edges between $u_2$ and $U_2$ are of color $i$. In this way, we remove the $K_{1,s}$ of color $i$ centered at $u_0$ from $X_i$, while adding two new $K_{1,s}$ of color $i$ centered at $u_1$ and $u_2$, respectively. This results in $n_i'+1$ vertex-disjoint $K_{1,s}$ of color $i$, which contradicts the selection of $X_i$.

If $u_1$ and $u_2$ are not both leaves of $K_{1,s}$, without loss of generality, assume $u_1$ is the center of $K_{1,s}$. We can select a neighbor $w$ of $u_1$ in $V(G)\setminus X_i$ such that the edge between $u_1$ and $w$ is of color $i$. Replacing $u_2$ with $w$, we obtain a $K_{1,s}$ of color $i$ centered at $u_1$. Next, we select $s$ neighbors of $u_2$ from $V(G)\setminus (X_i\cup\{w\})$, denoted by the set $U_2$, such that all edges between $u_2$ and $U_2$ are of color $i$. This again results in $n_i'+1$ vertex-disjoint $K_{1,s}$ of color $i$, leading to a contradiction with the selection of $X_i$.
\end{proof}
  
  By the above claim, for each $K_{1,s}$ of color $i$ in $X_i$, we can remove one vertex such that every remaining vertex has at most $2s-2$ edges of color $i$ connecting to $V(G)\setminus X_i$. Thus, $n_i'$ vertices are removed from $X_i$, and the remaining vertex set is denoted by $X_i'$.
  
  Let $G'$ be the complete graph induced by $(V(G)\setminus(X_2\cup\cdots\cup X_k))\cup(X_2'\cup\cdots\cup X_k')$. It is easy to verify that \[|G'|\ge n\,.\] For each $i\in[2,k]$, every vertex in $X_i'$ has at most $2s-2$ edges of color $i$ connecting to $V(G')\setminus X_i'$, and every vertex in $V(G')\setminus X_i'$ has at most $s-1\le2s-3$ edges of color $i$ within $G'\setminus X_i'$. Therefore, in $G'$, every vertex has at most 
  \[
  (2s-2)+(|X_i'|-1)\le(2s-2)+(n_i's-1)\le(n_i+1)s-3
  \]
  edges of color $i$. Consequently, every vertex in $G'$ has at least
  \begin{equation}\label{inequ1}
  (|G'|-1)-\sum_{i=2}^{k}((n_i+1)s-3)\ge|G'|/2
  \end{equation}
  edges of color $1$. The inequality holds because $|G'|/2\ge n/2\ge\sum_{i=2}^{k}(n_i+1)s-3k+4$. Furthermore, equality in (\ref{inequ1}) holds if and only if $|G'|/2=n/2=\sum_{i=2}^{k}(n_i+1)s-3k+4$.
  
  Thus, by Lemma~\ref{lem:Bondy}, whether $|G'|=n$ or $|G'|>n$, we can always find a $C_n$ of color $1$ in $G'$. This completes the proof.
  \end{proof}

\subsection{Proof of Theorem~\ref{thm:main5}}
Theorem~\ref{thm:main5} can be reduced to Lemmas~\ref{lem:2P3} and~\ref{lem:3P3} below. Moreover, Lemmas~\ref{lem:Dirac} and~\ref{lem:edge2P3} help streamline the proof.

\begin{lemma}[Dirac~\cite{Dirac1952}]\label{lem:Dirac}
  If $G$ is a $2$-connected graph, then it contains a cycle of length at least $\min\{2\delta(G),\,|V(G)|\}$.
\end{lemma}

\begin{lemma}\label{lem:edge2P3}
Let $G$ be a connected graph with $e(G)\ge 11$ and with at most one vertex of degree at most $2$. Then $G$ contains $2P_3$ as a subgraph.
\end{lemma}

\begin{proof}
If $\Delta(G)\le 4$, then the result follows from Lemma~\ref{edgedegree}. 
If $\Delta(G)\ge 5$, then the result follows from Lemma~\ref{degree5}.
\end{proof}

\begin{lemma}\label{lem:2P3}
\[
R_k(2P_3, P_3, \ldots, P_3)=
\begin{cases}
6, & \text{if}\ k=2,\\[2pt]
k+3, & \text{if}\ k\ \text{is odd},\\[2pt]
k+2, & \text{if}\ k\ \text{is even and}\ k\ge 4.
\end{cases}
\]
\end{lemma}
\begin{proof}
  The lower bound follows immediately from Theorem~\ref{thm:main1}. Hence, it remains to prove the upper bound.
  
  When $k=2$, one can easily verify the more general statement that $R(tP_3, P_3)=3t$. Therefore, we assume $k\ge 3$ in the following.
  
  Let $G_1$ denote the subgraph induced by all edges of color $1$, and let $\kappa(G_1)$ represent the connectivity of $G_1$. For each $i\in[2,k]$, any vertex can be incident with at most one edge of color $i$, as otherwise a $P_3$ in color $i$ would be formed. Thus, when $k$ is odd, $\delta(G_1)\ge (k+2)-(k-1)=3$; when $k$ is even, $\delta(G_1)\ge (k+1)-(k-1)=2$.
  
  When $\kappa(G_1)\ge 2$, let $C_{\ell}$ be a longest cycle in $G_1$. By Lemma~\ref{lem:Dirac}, $\ell\ge 4$. If $\ell\ge 6$, then $C_{\ell}$ contains $2P_3$ as a subgraph, so the theorem holds trivially. If $\ell=5$, noting that $|G_1|\ge 6$ and $\kappa(G_1)\ge 2$, there exists at least one edge between $C_5$ and $G_1-V(C_5)$, say $xy$, where $x\in V(C_5)$ and $y\notin V(C_5)$. It is easy to verify that $G_1[V(C_5)\cup\{y\}]$ contains $2P_3$ as a subgraph. If $\ell=4$, noting that $|G_1|\ge 6$ and $\kappa(G_1)\ge 2$, there exists a matching of two edges between $C_4$ and $G_1-V(C_4)$, say $x_1y_1,x_2y_2$, where $x_1,x_2\in V(C_4)$ and $y_1,y_2\notin V(C_4)$. It is easy to verify that $G_1[V(C_4)\cup\{y_1,y_2\}]$ contains $2P_3$ as a subgraph.
  
  When $\kappa(G_1)=1$, let $v$ be a cut vertex of $G_1$. Then, $G_1-v$ has at least two connected components. If $G_1-v$ contains a $P_3$ as a subgraph, say $x_1x_2x_3$, let $u$ be a vertex that belongs to a different connected component from this $P_3$. Since $\delta(G_1)\ge 2$, $u$ has at least two neighbors, none of which belong to $\{x_1,x_2,x_3\}$. Therefore, $G_1$ contains $2P_3$ as a subgraph. If $G_1-v$ does not contain a $P_3$ as a subgraph, since $G_1-v$ has an odd number of vertices, there must exist a vertex $w$ that is isolated in $G_1-v$. Hence, $w$ has at most one neighbor in $G_1$, contradicting $\delta(G_1)\ge 2$.
  
  When $\kappa(G_1)=0$, since $\delta(G_1)\ge 2$, every connected component of $G_1$ contains a $P_3$ as a subgraph. Therefore, $G_1$ contains $2P_3$ as a subgraph.
\end{proof}
\begin{lemma}\label{lem:3P3}
\[
R_k(3P_3, P_3, \ldots, P_3)=
\begin{cases}
9, & \text{if}\ k\in\{2,3,4\},\\[2pt]
k+4, & \text{if}\ k\ \text{is odd and}\ k\ge 5,\\[2pt]
11, & \text{if}\ k=6,\\[2pt]
k+3, & \text{if}\ k\ \text{is even and}\ k\ge 8.
\end{cases}
\]
\end{lemma}

\begin{proof}
When $k\neq 6$, the lower bound follows from Theorem~\ref{thm:main1}. 

When $k=6$, we construct a $6$-edge-coloring of $K_{10}$ with vertex set $V(K_{10})=\{x_1,\ldots,x_5,y_1,\ldots,y_5\}$.

Let
\[
A=\{x_1,x_2,x_3,x_4,x_5\}
\qquad\text{and}\qquad
B=\{y_1,y_2,y_3,y_4,y_5\}.
\]
Color all edges inside $A$ and all edges inside $B$ with color $1$ (black). Thus the color-$1$ subgraph is $K_5\cup K_5$.

It remains to color the edges between $A$ and $B$. For each $t\in\{0,1,2,3,4\}$, let
\[
M_t=\{x_i\,y_{i+t}: i\in[5]\},
\]
where the subscripts are taken modulo $5$, with $y_0=y_5$. Then $M_0,M_1,M_2,M_3,M_4$ are five pairwise edge-disjoint perfect matchings of the complete bipartite graph $K_{5,5}$ with bipartition $(A,B)$. Color the edges of $M_t$ with color $t+2$.

This defines a $6$-edge-coloring of $K_{10}$. In color $1$, the monochromatic graph is $K_5\cup K_5$, and hence it contains no copy of $3P_3$, because each component has only five vertices and so can contain at most one vertex-disjoint copy of $P_3$. For each color $i\in\{2,3,4,5,6\}$, the monochromatic graph is a perfect matching, and therefore contains no copy of $P_3$.

Consequently, this coloring of $K_{10}$ contains neither a monochromatic copy of $3P_3$ in color $1$ nor a monochromatic copy of $P_3$ in any of the other five colors. Therefore, $R_6(3P_3,P_3,\ldots,P_3)\ge 11$.

\begin{figure}[htbp]
\centering
\begin{tikzpicture}[
    x=1cm,y=1cm,
    vertex/.style={circle,draw=black,fill=white,inner sep=2.1pt},
    lab/.style={font=\small},
    blackedge/.style={draw=black,line width=0.9pt},
    c2/.style={draw=blue!75!black,line width=1.05pt},
    c3/.style={draw=orange!90!black,line width=1.05pt},
    c4/.style={draw=red!85!black,line width=1.05pt},
    c5/.style={draw=green!60!black,line width=1.05pt},
    c6/.style={draw=pink!70!black,line width=1.05pt}
]

\def\R{1.35}

\coordinate (a1) at ({-4 + \R*cos(90)},{ 1.95 + \R*sin(90)});
\coordinate (a2) at ({-4 + \R*cos(18)},{ 1.95 + \R*sin(18)});
\coordinate (a3) at ({-4 + \R*cos(-54)},{1.95 + \R*sin(-54)});
\coordinate (a4) at ({-4 + \R*cos(-126)},{1.95 + \R*sin(-126)});
\coordinate (a5) at ({-4 + \R*cos(162)},{1.95 + \R*sin(162)});

\coordinate (b6)  at ({-4 + \R*cos(90)},{ -1.95 + \R*sin(90)});
\coordinate (b7)  at ({-4 + \R*cos(18)},{ -1.95 + \R*sin(18)});
\coordinate (b8)  at ({-4 + \R*cos(-54)},{-1.95 + \R*sin(-54)});
\coordinate (b9)  at ({-4 + \R*cos(-126)},{-1.95 + \R*sin(-126)});
\coordinate (b10) at ({-4 + \R*cos(162)},{-1.95 + \R*sin(162)});

\foreach \u/\v in {
a1/a2,a1/a3,a1/a4,a1/a5,a2/a3,a2/a4,a2/a5,a3/a4,a3/a5,a4/a5,
b6/b7,b6/b8,b6/b9,b6/b10,b7/b8,b7/b9,b7/b10,b8/b9,b8/b10,b9/b10}
{
    \draw[blackedge] (\u)--(\v);
}

\foreach \p in {a1,a2,a3,a4,a5,b6,b7,b8,b9,b10}
    \node[vertex] at (\p) {};

\node[lab,above]       at (a1) {$x_1$};
\node[lab,right]       at (a2) {$x_2$};
\node[lab,below right] at (a3) {$x_3$};
\node[lab,below left]  at (a4) {$x_4$};
\node[lab,left]        at (a5) {$x_5$};

\node[lab,above]       at (b6)  {$y_1$};
\node[lab,right]       at (b7)  {$y_2$};
\node[lab,below right] at (b8)  {$y_3$};
\node[lab,below left]  at (b9)  {$y_4$};
\node[lab,left]        at (b10) {$y_5$};

\node[font=\small] at (-4,4.05) {Color $1$: two disjoint copies of $K_5$};

\foreach \i/\x in {1/0,2/1.9,3/3.8,4/5.7,5/7.6}{
    \coordinate (u\i) at (\x, 2.55);
    \coordinate (v\i) at (\x, -2.55);
}

\foreach \i in {1,2,3,4,5}
    \draw[c2] (u\i)--(v\i);

\foreach \i/\j in {1/2,2/3,3/4,4/5,5/1}
    \draw[c3] (u\i)--(v\j);

\foreach \i/\j in {1/3,2/4,3/5,4/1,5/2}
    \draw[c4] (u\i)--(v\j);

\foreach \i/\j in {1/4,2/5,3/1,4/2,5/3}
    \draw[c5] (u\i)--(v\j);

\foreach \i/\j in {1/5,2/1,3/2,4/3,5/4}
    \draw[c6] (u\i)--(v\j);

\foreach \i in {1,2,3,4,5}{
    \node[vertex] at (u\i) {};
    \node[vertex] at (v\i) {};
}

\node[lab,above] at (u1) {$x_1$};
\node[lab,above] at (u2) {$x_2$};
\node[lab,above] at (u3) {$x_3$};
\node[lab,above] at (u4) {$x_4$};
\node[lab,above] at (u5) {$x_5$};

\node[lab,below] at (v1) {$y_1$};
\node[lab,below] at (v2) {$y_2$};
\node[lab,below] at (v3) {$y_3$};
\node[lab,below] at (v4) {$y_4$};
\node[lab,below] at (v5) {$y_5$};

\node[font=\small] at (3.8,4.05) {Colors $2$--$6$: a decomposition of $K_{5,5}$ into five perfect matchings};

\end{tikzpicture}
\caption{A $6$-edge-coloring of $K_{10}$ witnessing $R_6(3P_3,P_3,\ldots,P_3)\ge 11$.}
\label{fig:lowerbound-R6-3P3}
\end{figure}
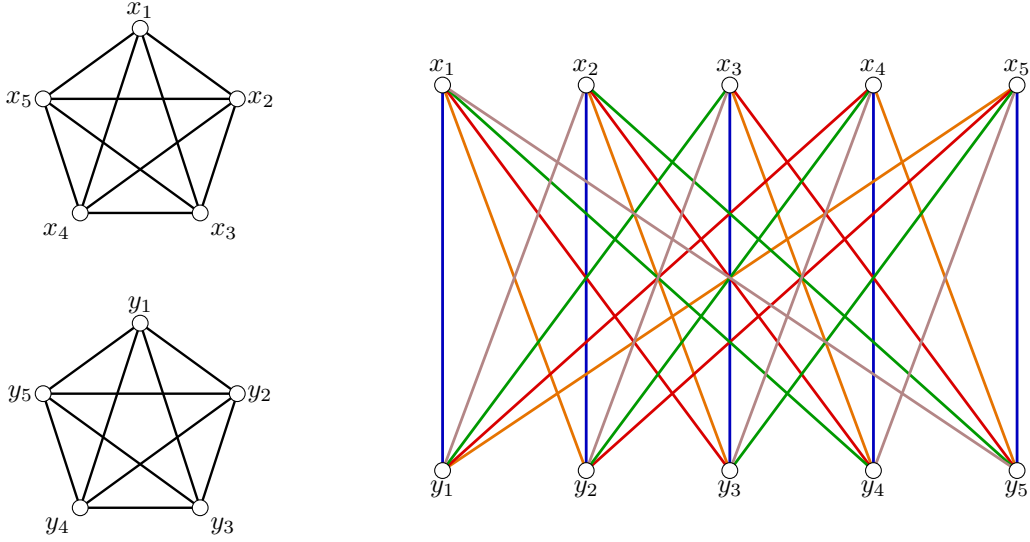

It remains to prove the upper bound.

For $k=2$, the stronger identity $R(tP_3,P_3)=3t$ is immediate. Hence we may assume that $k\ge 3$. Let
\[
N=
\begin{cases}
9, & \text{if } k\in\{3,4\},\\
k+4, & \text{if } k\ge 5 \text{ is odd},\\
11, & \text{if } k=6,\\
k+3, & \text{if } k\ge 8 \text{ is even}.
\end{cases}
\]
Consider an arbitrary $k$-edge-coloring of $K_N$. Suppose that there is no monochromatic copy of $P_3$ in any color $i\in[2,k]$. Let $G_1$ be the spanning subgraph of $K_N$ whose edge set consists of all edges of color $1$. We shall show that $G_1$ contains a copy of $3P_3$.

Since no color $i\in[2,k]$ contains a copy of $P_3$, each vertex is incident with at most one edge of color $i$ for every $i\in[2,k]$. Therefore
\[
\delta(G_1)\ge N-1-(k-1)=N-k\ge 3.
\]

We distinguish three cases according to the connectivity of $G_1$.

First suppose that $\kappa(G_1)\ge 2$. Let $C_\ell$ be a longest cycle in $G_1$. By Lemma~\ref{lem:Dirac}, we have $\ell\ge 6$. If $\ell\ge 9$, then $C_\ell$ clearly contains three vertex-disjoint copies of $P_3$. Thus we may assume that $\ell\in\{6,7,8\}$.

If $\ell=8$, then, since $|G_1|\ge 9$ and $G_1$ is $2$-connected, some vertex $y\in V(G_1)\setminus V(C_8)$ is adjacent to a vertex of $C_8$. It is then immediate that $G_1[V(C_8)\cup\{y\}]$ contains a copy of $3P_3$.

If $\ell=7$, then $|G_1|\ge 9$. Since $G_1$ is $2$-connected, there exist two vertices $y_1,y_2\in V(G_1)\setminus V(C_7)$ and two edges from $\{y_1,y_2\}$ to $C_7$. Together with the cycle $C_7$, these two vertices yield three vertex-disjoint copies of $P_3$. Hence $G_1$ contains a copy of $3P_3$.

It remains to consider the case $\ell=6$. Write
\[
C_6=x_1x_2x_3x_4x_5x_6x_1
\]
and put $S=V(G_1)\setminus V(C_6)$. Since $C_6$ contains two vertex-disjoint copies of $P_3$, we may assume that $G_1[S]$ contains no copy of $P_3$; otherwise, $G_1$ already contains a copy of $3P_3$. Thus $G_1[S]$ is a disjoint union of isolated vertices and edges. Moreover, $|S|=N-6$ is odd, and hence $G_1[S]$ contains an odd number of isolated vertices.

Let $y$ be an isolated vertex of $G_1[S]$. Since $\delta(G_1)\ge 3$, the vertex $y$ has at least three neighbors on $C_6$. By the maximality of $C_6$, no two of these neighbors are consecutive on $C_6$. Hence
\[
N_{C_6}(y)=\{x_1,x_3,x_5\}
\quad\text{or}\quad
N_{C_6}(y)=\{x_2,x_4,x_6\}.
\]
If $G_1[S]$ contains at least three isolated vertices, say $y_1,y_2,y_3$, then two of them have the same neighborhood on $C_6$. Without loss of generality, assume
\[
N_{C_6}(y_1)=N_{C_6}(y_2)=\{x_1,x_3,x_5\}.
\]
If $N_{C_6}(y_3)=\{x_1,x_3,x_5\}$, then $y_1x_1y_2, y_3x_3x_2, x_4x_5x_6$
are three vertex-disjoint copies of $P_3$. If
$N_{C_6}(y_3)=\{x_2,x_4,x_6\}$, then $y_1x_1y_2, y_3x_2x_3, x_4x_5x_6$ are three vertex-disjoint copies of $P_3$. Hence $G_1$ contains a copy of $3P_3$.

We may therefore assume that $G_1[S]$ has exactly one isolated vertex, say $y_1$. Then, since $|S|\ge 3$ is odd, there is an edge $y_2y_3$ in $G_1[S]$. Without loss of generality, assume
\[
N_{C_6}(y_1)=\{x_1,x_3,x_5\}.
\]
Since $\delta(G_1)\ge 3$, each of $y_2$ and $y_3$ has at least two neighbors on $C_6$. Moreover, by the maximality of $C_6$, the neighbors of each of $y_2$ and $y_3$ on $C_6$ must be pairwise nonconsecutive. Moreover, if $x_i\in N_{C_6}(y_2)$ and $x_j\in N_{C_6}(y_3)$, then either $x_i=x_j$, or every path on $C_6$ between $x_i$ and $x_j$ has length at least three; otherwise, replacing such a path with $x_i y_2 y_3 x_j$ would produce a cycle longer than $C_6$. It follows that $N_{C_6}(y_2)\cup N_{C_6}(y_3)$ is one of
\[
\{x_1,x_4\},\qquad \{x_2,x_5\},\qquad \{x_3,x_6\}.
\]
In each case, however, we obtain a cycle longer than $C_6$. For instance, if $N_{C_6}(y_2)\cup N_{C_6}(y_3)=\{x_1,x_4\}$, then $x_1x_2x_3y_1x_5x_4y_3y_2x_1$ is a cycle of length $8$, contradicting the choice of $C_6$ as a longest cycle. The other two cases are analogous. This contradiction completes the case $\kappa(G_1)\ge 2$.

Next suppose that $\kappa(G_1)=1$. Let $v$ be a cut vertex of $G_1$. Since $\delta(G_1)\ge 3$, every vertex of $G_1-v$ has degree at least $2$ in its component. Thus every component of $G_1-v$ contains a copy of $P_3$. If $G_1-v$ has at least three components, then choosing one copy of $P_3$ from each of three components gives a copy of $3P_3$. Hence we may assume that $G_1-v$ has exactly two components, say $A$ and $B$. Moreover, if either component contains a copy of $2P_3$, then together with a copy of $P_3$ in the other component we obtain a copy of $3P_3$. Therefore, we may assume that neither $A$ nor $B$ contains a copy of $2P_3$.

If $k\le 5$, then $N=9$ and $\delta(G_1)\ge 4$. Hence each of the two components of $G_1-v$ has at least four vertices, and since they have altogether eight vertices, both components have order four. The minimum-degree condition forces both components to be copies of $K_4$, and $v$ is adjacent to every vertex of both components. It is then immediate that $G_1$ contains a copy of $3P_3$.

Now assume that $k\ge 6$. Then $N\ge 11$. Without loss of generality, let $|A|\le |B|$, and put
\[
G_2=G_1-V(A).
\]
Then $|G_2|\ge 6$.

If $|G_2|\ge 8$, then the minimum-degree condition gives $e(G_2)\ge 11$. Moreover, every vertex of $G_2$ other than possibly $v$ has degree at least $3$ in $G_2$. Hence, by Lemma~\ref{lem:edge2P3}, the graph $G_2$ contains a copy of $2P_3$. Together with a copy of $P_3$ in $A$, this gives a copy of $3P_3$ in $G_1$.

It remains to consider the case $6\le |G_2|\le 7$. Let
\[
G_3=G_1-V(B).
\]
Since $|G_3|=N+1-|G_2|$ and $|G_3|\le |G_2|$, we have $5\le |G_3|\le 7$. Choose a longest path $P'$ in $G_2$ with $v$ as an end-vertex, and a longest path $P''$ in $G_3$ with $v$ as an end-vertex. By the minimum-degree condition, it is straightforward to verify that both $P'$ and $P''$ have length at least $4$. Therefore, joining them at their common end-vertex $v$, we obtain a copy of $P_9$. This path contains a copy of $3P_3$.

Finally suppose that $\kappa(G_1)=0$. If $k\le 5$, then $N=9$ and $\delta(G_1)\ge 4$, which is impossible for a disconnected graph on nine vertices. Hence $k\ge 6$, and so $N\ge 11$. Since $\delta(G_1)\ge 3$, every component of $G_1$ contains a copy of $P_3$. If $G_1$ has at least three components, then we immediately obtain a copy of $3P_3$. Thus we may assume that $G_1$ has exactly two components, say $A$ and $B$. If either component contains a copy of $2P_3$, then together with a copy of $P_3$ in the other component we obtain a copy of $3P_3$. Therefore, assume that neither component contains a copy of $2P_3$. Without loss of generality, let
\[
|B|\ge \left\lceil \frac{N}{2}\right\rceil\ge 6.
\]

If $|B|\ge 7$, then $\delta(B)\ge 3$ implies
\[
e(B)\ge \frac{3|B|}{2}\ge \frac{21}{2},
\]
and hence $e(B)\ge 11$. By Lemma~\ref{lem:edge2P3}, the graph $B$ contains a copy of $2P_3$.

If $|B|=6$, then $\delta(B)\ge 3$ implies that $B$ is $2$-connected. Thus, by Lemma~\ref{lem:Dirac}, the graph $B$ contains a copy of $C_6$, and consequently contains a copy of $2P_3$.

In both cases, together with a copy of $P_3$ in $A$, this gives a copy of $3P_3$ in $G_1$.

In all cases, $G_1$ contains a copy of $3P_3$. Therefore every $k$-edge-coloring of $K_N$ contains either a copy of $3P_3$ in color $1$ or a copy of $P_3$ in some color $i\in[2,k]$. This proves the desired upper bound and completes the proof.
\end{proof}

\section{Gallai--Ramsey numbers of disjoint unions of cherries}\label{secforest1}

In this section, we present the following crucial theorem. Theorem~\ref{Wuconjecture1} can be easily derived from this result.

\begin{theorem}\label{main}
	Let $(G, \tau)$ be a Gallai $k$-colored complete graph with colors $1, 2, \ldots, k$. Assume that $(G, \tau)$ contains a monochromatic copy of $P_3$ in color $i$ for each $i\in [q]$, while no monochromatic copy of $P_3$ in color $j$ for any $j\in [q+1,k]$. Furthermore, let $n_1, n_2, \ldots, n_q$ be positive integers, and define $n_0$ to be $\max\{n_1,\ldots,n_q\}$. If \begin{equation}\label{equation_1}
    |G|\ge 2n_0+\sum_{i=1}^{q}n_i-q+1\,,
  \end{equation}
  then there exists an $i\in[q]$ such that $G$ contains a monochromatic copy of $n_iP_3$ in color $i$.
\end{theorem}

Let us first examine how the above theorem leads to the statement of Theorem~\ref{Wuconjecture1}, which we will now restate.
 \begin{corollary}
	Let $k, n_1, \ldots, n_k$ be positive integers and $n=\max\{n_1,\ldots,n_k\}$. We have 
	\[GR(n_1P_3, \ldots, n_kP_3)=2n+\sum_{i=1}^{k}n_i-k+1\,\]
\end{corollary}

\begin{proof}
  The lower bound has already been established by Wu, Magnant, Nowbandegani, and Xia~\cite{Wu2019}. Here, we will focus solely on the upper bound.

  Let $(G, \tau)$ represent a Gallai $k$-colored complete graph with $|G|=2n+\sum_{i=1}^{k}n_i-k+1$, where the colors are denoted as $1, 2, \ldots, k$. It is observed that for certain colors, $G$ contains a monochromatic copy of $P_3$, while for other colors, $G$ does not. Consequently, we can assume that $G$ contains a monochromatic copy of $P_3$ in color $i$ for each $i\in [q]$, while no monochromatic copy of $P_3$ exists in color $j$ for any $j\in [q+1,k]$. Note that this assumption can easily be obtained through the permutation of colors, and $0\le q\le k$. Since $G$ contains at least three vertices, and any Gallai coloring of three vertices must contain a monochromatic $P_3$, it follows that $q\ge 1$. Let $n_0=\max\{n_1,\ldots,n_q\}$. Clearly, $|G|\ge 2n_0+\sum_{i=1}^{q}n_i-q+1$, satisfying the conditions of Theorem~\ref{main}. Therefore, $G$ contains a monochromatic copy of $n_iP_3$ in color $i$ for some $i \in [k]$.
\end{proof}

Now we prove Theorem~\ref{main}. When $q=1$, we have $|G|\ge 3n_1$. Thus, $G$ contains $n_1$ pairwise disjoint triangles. Each triangle has at least two edges colored with color 1. Therefore, $G$ contains an $n_1P_3$ of color $1$ as a subgraph.

When $q=2$, we require the following lemma.

\begin{lemma}[Faudree and Schelp~\cite{Faudree1976}]\label{Faudree}
  $R(n_1P_3,n_2P_3)=\max\{3n_1+n_2-1, 3n_2+n_1-1\}$.
\end{lemma}

When $q=2$, we have $|G|\ge \max\{3n_1+n_2-1, 3n_2+n_1-1\}$. We consider colors 2, 3, ..., $k$ as a single color, which transforms this into a two-edge coloring of $G$. According to Lemma~\ref{Faudree}, $G$ either contains an $n_1P_3$ of color $1$, which completes the proof, or it contains a non-$1$-colored $n_2P_3$, meaning each $P_3$ has no edges colored with color 1. In the second case, each $P_3$ is contained in a triangle. Since a triangle in a Gallai coloring must have at least two edges of the same color, this triangle must have two edges colored with color 2. Consequently, $G$ contains an $n_2P_3$ of color $2$. This completes the proof for the case $q=2$. Based on the above proof, we will assume \[q\ge 3\,.\]

Assuming $n_i=1$ for some $i \in [q]$, then the proposition naturally holds according to the conditions of Theorem~\ref{main}. Therefore, we assume \[n_i\ge 2\ \text{for each}\ i \in [q]\,.\]

We shall employ the method of minimal counterexamples. Suppose to the contrary that there is such a graph $G$ that satisfies the conditions of the theorem, yet fails to contain any monochromatic $n_iP_3$ in color $i$ for each $i\in [q]$. Among all counterexamples to the theorem, we select $q$ as small as possible; then we choose $n_0$ as small as possible; subject to the choices of $q$ and $n_0$, we choose $\sum_{i=1}^{q}n_i$ as small as possible.

Let $\{V_1, \ldots, V_p\}$ be a Gallai partition of $(G-X, \tau)$ with $p$ being as small as possible and $p \ge 2$. We may assume that $|V_1| \le \cdots \le |V_p|$. According to Theorem~\ref{Gallai}, all edges within the reduced graph of $(G-X, \tau)$ are colored using at most two colors, say red and blue. Without loss of generality, let us assume that the color 1 is red, while the color 2 is blue.

\begin{claim}\label{3n0}
  Assume that $V(G)$ contains two disjoint non-empty vertex sets $A$ and $B$. Furthermore, let all edges between $A$ and $B$ be colored with the same color $i$ for some $i\in [2]$, and let $|A|+|B| \ge 2n_0+n_i$. Then $(G, \tau)$ contains an $n_iP_3$ of color $i$.
\end{claim}

\begin{proof}
We will only prove the case where $i=2$, as the other case follows similarly. Without loss of generality, assume $|A| \le |B|$. Thus, we have $|B| \ge n_0 + \frac{n_2}{2}$. 

Select any vertex $w \in A$. By the induction hypothesis, $G-w$ contains a blue $(n_2-1)P_3$ as a subgraph, which we denote as $H_1$. If $|B\setminus V(H_1)| \ge 2$, then by choosing two vertices from $B\setminus V(H_1)$, the induced subgraph along with $w$ will contain a blue $P_3$. Together with $H_1$, this forms a blue $n_2P_3$, leading to a contradiction. Therefore, we conclude that $|B\setminus V(H_1)| \le 1$. 

From $|A| + |B| \ge 2n_0 + n_2$ and $|V(H_1)| = 3n_2 - 3$, we deduce that $|A\setminus V(H_1)| + |B\setminus V(H_1)| \ge 3$. 

If $|B\setminus V(H_1)| = 1$, then $|A\setminus V(H_1)| \ge 2$. By selecting one vertex from $B\setminus V(H_1)$ and two vertices from $A\setminus V(H_1)$, the induced subgraph on these three vertices will contain a blue $P_3$, which together with $H_1$ will again form a blue $n_2P_3$, leading to a contradiction. 

If $|B\setminus V(H_1)| = 0$, then $|A\setminus V(H_1)| \ge 3$. Let us denote the vertices in $A\setminus V(H_1)$ as $w_1, w_2, w_3$. Since $|B| \ge n_0 + \frac{n_2}{2}$, among the $n_2 - 1$ blue $P_3$ in $H_1$, there must be at least one blue $P_3$, denoted as $x_1x_2x_3$, with at least two of its vertices in $B$. Without loss of generality, assume $x_1, x_2 \in B$. We can then remove the blue $P_3 = x_1x_2x_3$ from $H_1$ and add two new blue $P_3$: $w_1x_1w_2$ and $w_3x_2x_3$. This results in a blue $n_2P_3$, leading to a contradiction. If instead $x_1, x_3 \in B$ or $x_2, x_3 \in B$, the same technique can be applied to obtain a blue $n_2P_3$.
\end{proof}  

Notice that for all $i\in[p-1]$, $V_i$ is either red-complete to $V_p$ or blue-complete to $V_p$. Let
\[
\begin{split}
 R=\bigcup\{V_i\colon\, i\in[p-1]\text{ and }V_i \text{ is red-complete to } V_p\} \text{ and}\\
 B=\bigcup\{V_i\colon\, i\in[p-1]\text{ and }V_i \text{ is blue-complete to } V_p\}.
\end{split}
\] We have the following claim.

\begin{claim}\label{RBsize}
  $|R|\ge 3$ and $|B|\ge 3$.
\end{claim}

\begin{proof}
We only need to prove the former, as the latter can be shown similarly. Using proof by contradiction, assume that $|R|\le 2$. Then, we have 
\[|V_p|+|B|\ge 2n_0+\sum_{i=1}^{q}(n_i-1)-1\ge 2n_0+n_2.\] 
According to Claim~\ref{3n0}, $G$ contains a blue $n_2P_3$, which leads to a contradiction and thus confirms that $|R|\ge 3$.
\end{proof}

\begin{claim}\label{Vp}
  $|V_p|\le 2n_0-2$.
\end{claim}

\begin{proof}
Assuming that $|V_p|\ge 2n_0-1$, we have 
\[
2|V_p|+|R|+|B|\ge 4n_0+\sum_{i=1}^{q}n_i-q\ge 4n_0+n_1+n_2-1.
\]
By the pigeonhole principle, either $|V_p|+|R|\ge 2n_0+n_1$ or $|V_p|+|B|\ge 2n_0+n_2$. According to Claim~\ref{3n0}, the graph $G$ contains either a red $n_1P_3$ or a blue $n_2P_3$. This contradiction leads to the conclusion that $|V_p|\le 2n_0-2$.
\end{proof}

A graph is said to be \emph{properly colored} if adjacent edges are colored with different colors. We have the following important lemma, which can also be viewed as a result related to proper Ramsey numbers.

\begin{lemma}\label{proper}
  For $s\ge 2$, in any red-blue edge coloring of the complete graph on $s$ vertices, there exists either a monochromatic $K_{1,s-1}$ or a properly colored even cycle.
\end{lemma}

\begin{proof}
When $2\le s\le 3$, it is clear that a monochromatic $K_{1,s-1}$ can be found. For $s\ge 4$, assuming there is no monochromatic $K_{1,s-1}$, we will locate a properly colored even cycle within $K_s$. We start by finding the longest properly colored path in $K_s$. Let this path be denoted as $v_1v_2\cdots v_t$, where $4\le t\le s$. The condition $t\ge 4$ is easily verified; it essentially corresponds to a result related to proper Ramsey numbers: $PR(K_{1,s-1},P_4)=s$ for $s\ge 4$ (\cite{Olejniczak2019}, Proposition 2.5.2).

If the colors of the two edges $v_1v_2$ and $v_{t-1}v_t$ are different, we can assume without loss of generality that $v_1v_2$ is red and $v_{t-1}v_t$ is blue. The edge $v_1v_t$ can be either red or blue. By symmetry, we may assume that $v_1v_t$ is red. Notably, all edges between $v_1$ and any vertex outside this longest path must be red; otherwise, a longer properly colored path would exist, contradicting our previous assumption. To avoid forming a red $K_{1,s-1}$, vertex $v_1$ must be connected by a blue edge to some vertex on this path. Without loss of generality, assume $v_1v_i$ is blue, where $3\le i\le t-1$. If $v_iv_{i+1}$ is red, then $v_1v_iv_{i+1}\cdots v_tv_1$ forms a properly colored even cycle. Conversely, if $v_iv_{i+1}$ is blue, then $v_1v_2\cdots v_iv_1$ forms a properly colored even cycle.

If the colors of the edges $v_1v_2$ and $v_{t-1}v_t$ are the same, we can assume without loss of generality that both $v_1v_2$ and $v_{t-1}v_t$ are red. Since this path has an even number of vertices, we have either $t=4$ or $t\ge 6$. If the edge $v_1v_t$ is blue, then $v_1v_2\cdots v_tv_1$ forms a properly colored even cycle. Thus, we assume $v_1v_t$ is red. 

When $t=4$, since each vertex is connected to both red and blue edges, it follows that $v_1v_3$ and $v_2v_4$ must be blue. Consequently, $v_1v_3v_4v_2v_1$ forms a properly colored $C_4$. 

For $t\ge 6$, if $v_1v_{t-1}$ is red, then there exists some $3\le i\le t-2$ such that $v_1v_i$ is blue. Following the same argument as before, either $v_1v_iv_{i+1}\cdots v_{t-1}v_1$ or $v_1v_2\cdots v_iv_1$ forms a properly colored even cycle. Therefore, we assume $v_1v_{t-1}$ is blue. By symmetry, we may also assume that $v_2v_t$ is blue. Thus, $v_1v_2v_tv_{t-1}v_1$ forms a properly colored $C_4$.
\end{proof}

In the final part, we prove two key claims. The first is that $n_i < n_0$ for each $i \in [3,q]$. The second is that we can find four vertices such that removing these four vertices will eliminate one red $P_3$ and one blue $P_3$. We can then apply the inductive hypothesis to the remaining complete graph to complete our proof.

\begin{claim}\label{nin0}
  $n_i < n_0$ for each $i \in [3,q]$.
\end{claim}

\begin{proof}
  We use proof by contradiction. Assume $n_3=n_0$, which will lead to a contradiction. By the inductive hypothesis, the graph $G$ contains an $(n_0-1)P_3$ of color $3$, denoted as $H_2$. Since the edges between different parts $V_i$ and $V_j$ can only be red or blue, all three vertices of any $P_3$ of color $3$ must lie within the same part. If all vertices of $H_2$ are contained in a single part $V_i$, then $|V_p| \ge |V_i| \ge 3n_0 - 3 > 2n_0 - 2$, which contradicts Claim~\ref{Vp}. 

  Assume there are $s$ parts that intersect with $V(H_2)$, where $s \ge 2$. Now, we consider only these $s$ parts. In its reduced graph, we denote the vertices corresponding to these $s$ parts as $v_1, v_2, \ldots, v_s$. In the subgraph $K_s$, there are two cases: either there exists a monochromatic $K_{1,s-1}$, or there does not. We will discuss both scenarios.

  For the first case, let the part corresponding to the center of $K_{1,s-1}$ be denoted as $V'$, and let the union of the other $s-1$ parts that intersect with $V(H_2)$ be denoted as $V''$. Without loss of generality, we can assume that the edges between $V'$ and $V''$ are red. Let $V'$ contain $x$ disjoint $P_3$s of color $3$, then $V''$ contains $n_0-1-x$ disjoint $P_3$s of color $3$.

  It can be shown that $|V'| \ge 3x + 1$. This is because, if $|V'| = 3x$, removing one vertex $u_1$ from $V'$ means, by the inductive hypothesis, that $G - u_1$ contains an $(n_0 - 1)P_3$ of color $3$ as a subgraph, denoted as $H_3$. Since $|V'-\{u_1\}|=3x-1$, there can be at most $x - 1$ $P_3$s of color $3$ contained in $V'-u_1$. In the graph $G$, we can remove the $P_3$s from $H_3$ that are contained in $V'-\{u_1\}$ and replace them with the $x$ $P_3$s contained in $H_2$ that are in $V'$. This results in at least $n_0$ disjoint $P_3$s of color $3$, leading to a contradiction. Similarly, we can prove that $|V''| \ge 3(n_0 - 1 - x) + 1$. Thus, we have $|V'| + |V''| \ge 3n_0 - 1$.

  If $|V'| + |V''| \ge 2n_0 + n_1$, then according to Claim~\ref{3n0}, $G$ contains a monochromatic (i.e., red) $n_1P_3$, which leads to a contradiction. Therefore, we assume $n_1 = n_0$ and that $|V'| + |V''| = 3n_0 - 1$.

  In this case, it is easy to compute that there are at least $n_2 + n_0 - 1$ vertices remaining in $V(G) \setminus (V' \cup V'')$. If a vertex $w_1$ in $V(G) \setminus (V' \cup V'')$ connects to $V'$ with red edges, we can treat $V'$ as $A$ and $V'' \cup \{w_1\}$ as $B$. Utilizing Claim~\ref{3n0}, we find that $G$ contains a red $n_1P_3$, which leads to a contradiction. Hence, all edges between $V(G) \setminus (V' \cup V'')$ and $V'$ must be blue. If $|V'| \ge n_0 + 1$, then $|V(G) \setminus V''| \ge 2n_0 + n_2$, again utilizing Claim~\ref{3n0}, resulting in $G$ containing a blue $n_2P_3$, which is also a contradiction. Therefore, we have $|V'| \le n_0$ and $|V''| \ge 2n_0 - 1$.

  If there are $n_2$ vertices in $V(G) \setminus (V' \cup V'')$, each connecting to $V' \cup V''$ with at least $2n_2$ blue edges, we can easily find a blue $n_2P_3$, where each $P_3$'s center lies in $V(G) \setminus (V' \cup V'')$. Thus, there are at least $n_0$ vertices in $V(G) \setminus (V' \cup V'')$, denoted as the set $W$, where each vertex connects to $V' \cup V''$ with at most $2n_2 - 1$ blue edges. Given that $|V'| + |V''| = 3n_0 - 1$ and that all edges between $W$ and $V'$ are blue, each vertex in $W$ connects to $V''$ with at least $n_0$ red edges.

  We remove a vertex $u_2$ from $V'$. By the inductive hypothesis, $G - u_2$ contains a red $(n_0 - 1)P_3$ as a subgraph, which we denote as $H_4$. Notice that $|V' \setminus V(H_4)| + |V'' \setminus V(H_4)| \ge 2$. If $|V'' \setminus V(H_4)| \ge 2$, then selecting any two vertices from $V'' \setminus V(H_4)$ along with $u_2$ forms a red $P_3$. Combining this $P_3$ with $H_4$ results in a red $n_0P_3$, leading to a contradiction. Therefore, we have
\[|V'' \setminus V(H_4)| \le 1\,.\]
Thus, either there is one vertex in both $V'$ and $V''$ not in $H_4$, or there are two vertices in $V'$ not in $H_4$. One of these vertices must be $u_2$, and we denote the other vertex as $u_3$.

Since $|W| \ge n_0$ and $|H_4|=3n_0-3$, with at least $2n_0-2$ vertices of $H_4$ already in $V''$, there exists a vertex in $W$ not included in $H_4$, denoted as $w_2$. If $(N_R(w_2) \cap V'') \not\subseteq V(H_4)$, based on the assumption from the previous paragraph, the vertex $u_3$ is the one in $N_R(w_2) \cap V''$ not contained in $V(H_4)$. Therefore, $w_2u_3u_2$ forms a red $P_3$ disjoint from $H_4$, and combining it with $H_4$ forms a red $n_1P_3$, leading to a contradiction. 

If $(N_R(w_2) \cap V'') \subseteq V(H_4)$, since $w_2$ is adjacent to at least $n_0$ vertices in $V''$ with red edges, there must be a red $P_3$ in $H_4$ that contains at least two vertices in $N_R(w_2) \cap V''$. Without loss of generality, let $x_1x_2x_3$ be a red $P_3$ in $H_4$, where $x_1, x_2 \in N_R(w_2) \cap V''$ (or $x_1, x_3 \in N_R(w_2) \cap V''$). Then, by removing $x_1x_2x_3$ from $H_4$ and adding a red $P_3$ induced by $w_2, x_2, x_3$ and another red $P_3$ induced by $x_1, u_2, u_3$, we again obtain a red $n_1P_3$, leading to a contradiction.  

For the second case, according to Lemma~\ref{proper}, there is a properly colored even cycle, denoted as $C_{2\ell}$, in the reduced subgraph $K_s$. In graph $G$, we take a $P_3$ of color $3$ from each part corresponding to the vertices of $C_{2\ell}$. In total, we extract $6\ell$ vertices. These vertices include $2\ell$ disjoint $P_3$s of color $3$, $2\ell$ disjoint red $P_3$s, and $2\ell$ disjoint blue $P_3$s. If $n_i \le 2\ell$ for some $i \in [3]$, the proof is complete. Therefore, we assume $n_i > 2\ell$ for each $i \in [3]$. After removing these $6\ell$ vertices from $G$, the remaining graph has at least 
\[
2n_0 + \sum_{i=1}^{3}(n_i-2\ell) + \sum_{i=3}^{q}n_i - q + 1\,.
\]
By the induction hypothesis, the new graph contains $n_i-2\ell$ disjoint $P_3$s of color $i$ for some $i \in [3]$. Combining these with the $6\ell$ removed vertices, we can find $n_i$ disjoint $P_3$s of color $i$ in $G$ for some $i \in [3]$. This completes our proof.
\end{proof}

\begin{claim}\label{claimK4}
  There does not exist a subgraph $K_4$ in $G$ that contains both a red $P_3$ and a blue $P_3$ as subgraphs.
\end{claim}
\begin{proof}
Suppose such a $K_4$ exists in the graph $G$. After removing these four vertices from $G$, the number of vertices in the resulting complete graph is at least
\[
  |V(G)|-4\ge 2(n_0-1)+(n_1-2)+(n_2-2)+\sum_{i=3}^{q}(n_i-1)+1\,.
\]
By Claim~\ref{nin0}, $n_i<n_0$ for $i \in [3,q]$, so $n_0-1 = \max\{n_1-1, n_2-1, n_3, n_4, \ldots , n_q\}$. Here, $n_0-1$, $n_1-1$, and $n_2-1$ replace $n_0$, $n_1$, and $n_2$ in Equation~(\ref{equation_1}), respectively. By the induction hypothesis, the new complete graph must contain either a red $(n_1-1)P_3$ or a blue $(n_2-1)P_3$. Together with the $K_4$ that was removed, this would result in finding a red $n_1P_3$ or a blue $n_2P_3$ in $G$, which leads to a contradiction.
\end{proof}

Note that the edges connecting $R$ and $B$ can only be red or blue. We first assume that there exists a vertex $v$ in $B$ that has at least two blue edges connecting to $R$. Then, by choosing any vertex from $V_p$ and any two vertices from $N_B(v) \cap R$, together with vertex $v$, the induced subgraph forms a $K_4$, which contains both a red $P_3$ and a blue $P_3$ as subgraphs. This contradicts Claim~\ref{claimK4}. Therefore, each vertex in $B$ has at most one blue edge connecting to $R$. By Claim~\ref{RBsize}, there exists a red-blue edge-colored $K_{3,3}$ between $R$ and $B$. Hence, there must exist a vertex $u$ in $R$ that has at least two red edges connecting to $B$. Now, by choosing any vertex from $V_p$ and any two vertices from $N_R(u) \cap B$, together with vertex $u$, the induced subgraph contains both a red $P_3$ and a blue $P_3$ as subgraphs. This again contradicts Claim~\ref{claimK4}. Thus, the proof of the theorem is complete.

\subsection*{Acknowledgements}

This research was supported by the National Key R\&D Program of China (Grant No.~2024YFA1013900) and by the National Natural Science Foundation of China (Grant No.~12471327).

\end{document}